\documentclass[11pt]{amsart}

\usepackage{amsmath}
\usepackage{fullpage}
\usepackage{xspace}
\usepackage[psamsfonts]{amssymb}
\usepackage[latin1]{inputenc}
\usepackage{graphicx,color}
\usepackage[curve]{xypic} 
\usepackage{hyperref}
\usepackage{graphicx}
\usepackage{amsmath}%
\usepackage{amsthm}%
\usepackage{amscd}
\usepackage{amsfonts}%
\usepackage{amssymb}%
\usepackage{graphicx}

\usepackage{mathrsfs}

\usepackage{tikz}
\usetikzlibrary{matrix,arrows}

\newtheorem{theorem}{Theorem}[section]

\newtheorem{conjecture}[theorem]{Conjecture}
\newtheorem{corollary}[theorem]{Corollary}

\newtheorem{lemma}[theorem]{Lemma}

\newtheorem{proposition}[theorem]{Proposition}

\theoremstyle{remark}

\numberwithin{equation}{section}

\newcommand{\hfrak}{\mathfrak{h}}

\newcommand{\Pcal}{\mathscr{P}}

\newcommand{\Z}{\mathbb{Z}}
\newcommand{\C}{\mathbb{C}}

\newcommand{\F}{\mathbb{F}}
\newcommand{\Q}{\mathbb{Q}}

\newcommand{\T}{\mathbb{T}}

\newcommand{\rk}{\mathrm{rank}\,}

  \DeclareFontFamily{U}{wncy}{}
    \DeclareFontShape{U}{wncy}{m}{n}{<->wncyr10}{}
    \DeclareSymbolFont{mcy}{U}{wncy}{m}{n}
    \DeclareMathSymbol{\Sha}{\mathord}{mcy}{"58}

\begin{document}
\title[]{A conjecture of Watkins for quadratic twists}

\author{Jose A. Esparza-Lozano}
\address{African Institute for Mathematical Sciences\newline 
\indent Rue KG590 ST,  Kigali, Rwanda}
\curraddr{Department of Mathematics\newline
\indent University of Michigan, Ann Arbor, USA}
\email[J. Esparza-Lozano]{josealanesparza@gmail.com}

\author{Hector Pasten}
\address{Pontificia Universidad Cat\'olica de Chile\newline
\indent Facultad de Matem\'aticas\newline
\indent 4860 Av. Vicu\~na Mackenna, Macul, RM, Chile}
\email[H. Pasten]{hector.pasten@mat.uc.cl}%

%\thanks{}
\thanks{J. E.-L. was supported by a Carroll L. Wilson Award and the MIT International Science and Technology Initiatives (MISTI). H. P. was supported by FONDECYT Regular grant 1190442.}
\date{\today}
%\dedicatory{}

\begin{abstract} Watkins conjectured that for an elliptic curve $E$ over $\Q$ of Mordell-Weil rank $r$, the modular degree of $E$ is divisible by $2^r$. If  $E$ has non-trivial rational $2$-torsion, we prove the conjecture for all the quadratic twists of $E$ by squarefree integers with sufficiently many prime factors. 
\end{abstract}

\maketitle

%%%%%%%%%%%%%%%%%%%%%%%%%%%%%%%%%%%%%%
%%%%%%%%%%%%%%%%%%%%%%%%%%%%%%%%%%%%%%
%%%%%%%%%%%%%%%%%%%%%%%%%%%%%%%%%%%%%%
%%%%%%%%%%%%%%%%%%%%%%%%%%%%%%%%%%%%%%
%%%%%%%%%%%%%%%%%%%%%%%%%%%%%%%%%%%%%%
%%%%%%%%%%%%%%%%%%%%%%%%%%%%%%%%%%%%%%

\section{Ranks and modular degree}

For an elliptic curve $E$ over $\Q$ of conductor $N$, the modularity theorem \cite{Wiles, TaylorWiles, BCDT} gives a non-constant morphism $\phi_E: X_0(N)\to E$ defined over $\Q$ where $X_0(N)$ is the modular curve associated to the congruence subgroup $\Gamma_0(N)\subseteq SL_2(\Z)$. We assume that $\phi_E$ has minimal degree and that it maps the cusp $i\infty$ to the neutral point of $E$. These requirements uniquely determine $\phi_E$ up to sign.  The \emph{modular degree} of $E$ is $m_E=\deg \phi_E$ and it has profound arithmetic relevance; for instance, polynomial bounds for its size in terms of $N$ are essentially equivalent to the $abc$ conjecture \cite{Frey, Murty}. 

The $2$-adic valuation is denoted by $v_2$.   Motivated by numerical data, Watkins \cite{Watkins} conjectured that $v_2(m_E)$ for an elliptic curve $E$ is closely related to the Mordell-Weil  rank of $E$ over $\Q$.
\begin{conjecture}[Watkins] For every elliptic curve $E$ over $\Q$ we have $\rk E(\Q)\le v_2(m_E)$.
\end{conjecture}
Dummigan \cite{Dummigan} showed that part of the conjecture would follow from strong $R=\T$ conjectures. Also, large part of Watkins' conjecture is proved for elliptic curves of odd modular degree \cite{CalegariEmerton, Yazdani, KazalickiKohen, KKerratum}, although it is not known whether there exist infinitely many elliptic curves of this kind \cite{SteinWatkins}. 

The goal of this note is to prove Watkins' conjecture unconditionally in several new cases. Let us introduce some notation. For an elliptic curve $E$ and a fundamental (quadratic) discriminant $D$, the quadratic twist of $E$ by $D$ is denoted by $E^{(D)}$. The Manin constant of $E$ is denoted by $c_E$ (cf. Section \ref{SecManin}). The number of distinct prime factors of an integer $n$ is $\omega(n)$.

\begin{theorem}\label{ThmMain} Let $E$ be an elliptic curve over $\Q$ of conductor $N$ with non-trivial rational $2$-torsion. Assume that $E$ has minimal conductor among its quadratic twists. If $D$ is a fundamental discriminant with $\omega(D)\ge 6+5\omega(N)-v_2(m_E/c_E^2)$, then Watkins' conjecture holds for $E^{(D)}$.
\end{theorem}

The quantity $6+5\omega(N)-v_2(m_E/c_E^2)$ is effectively computable and it can be read from existing tables of elliptic curves when $N$ is not too large, see for instance \cite{LMFDB}.

 For a positive integer $A$, it is a standard result of analytic number theory that the number of positive integers $n$ up to $x$ having $\omega(n)\le A$ is $O(x(\log \log x)^{A-1}/\log x)$. We deduce:

\begin{corollary} Let $E$ be an elliptic curve over $\Q$ with non-trivial rational $2$-torsion. There is an effective constant $\kappa(E)$ depending only on $E$ such that the number of fundamental discriminants $D$ with $|D|\le x$ such that Watkins' conjecture fails for $E^{(D)}$ is bounded by $O\left(x (\log\log x)^{\kappa(E)}/\log x\right)$.
\end{corollary}

Let us remark that in the cases where we prove Watkins' conjecture our argument actually shows that $v_2(m_{E^{(D)}})$ bounds the $2$-Selmer rank, which is a stronger version of Watkins' conjecture.

%%%%%%%%%%%%%%%%%%%%%%%%%%%%%%%%%%%%%%
%%%%%%%%%%%%%%%%%%%%%%%%%%%%%%%%%%%%%%
%%%%%%%%%%%%%%%%%%%%%%%%%%%%%%%%%%%%%%
%%%%%%%%%%%%%%%%%%%%%%%%%%%%%%%%%%%%%%
%%%%%%%%%%%%%%%%%%%%%%%%%%%%%%%%%%%%%%
%%%%%%%%%%%%%%%%%%%%%%%%%%%%%%%%%%%%%%

\section{Preliminaries}

\subsection{Faltings height} Let $E$ be an elliptic curve over $\Q$. We denote by $\omega_E$ a global Neron differential for $E$; it is unique up to sign. The Faltings height of $E$ (over $\Q$) is defined as certain Arakelov degree \cite{Faltings}, which in our case takes the simpler form \cite{SilvermanE}
\begin{equation}\label{EqnFH}
h(E)=-\frac{1}{2}\log \left(\frac{i}{2}\int_{E(\C)}\omega_E\wedge\overline{\omega_E}\right).
\end{equation}

Ramanujan's cusp form is $\Delta(z)=q\prod_{n=1}^\infty (1-q^n)^{24}$ where $q=\exp(2\pi i z)$, defined on the upper half plane $\hfrak=\{z\in \C : \Im(z)>0\}$. The modular $j$-function is normalized as $j(z)=q^{-1} + 744+...$

The global minimal discriminant of $E$ is denoted by $\Delta_E$.  If $\tau_E\in \hfrak$ satisfies that $j(\tau_E)$ is the $j$-invariant of $E$, then the Faltings height admits the expression \cite{Szpiro90, SilvermanE}
\begin{equation}\label{EqnhE}
h(E)=\frac{1}{12}\left(\log |\Delta_E| - \log \left| \Delta(\tau_E)\Im(\tau_E)^6\right|\right)-\log (2\pi).
\end{equation}
Given elliptic curves $E_1,E_2$ over $\Q$, let us define $\delta(E_1,E_2)=\exp(2h(E_1)-2h(E_2))$.
\begin{lemma}[Variation of $h(E)$ under quadratic twist] \label{LemmaTwist} Let $E_1$  be an elliptic curve over $\Q$ and let $E_2$ be a quadratic twist of $E_1$. Then $\delta(E_1,E_2)$ is a rational number and it satisfies $|v_2(\delta(E_1,E_2))|\le 3$.
\end{lemma}
\begin{proof} We use \eqref{EqnhE} for both $E_1$ and $E_2$.  The elliptic curves are isomorphic over $\C$, so we can take $\tau_{E_1}=\tau_{E_2}$ which gives $\delta(E_1,E_2)=|\Delta_{E_1}/\Delta_{E_2}|^{1/6}$. The result follows from explicit formulas for the variation of the minimal discriminant under quadratic twists, cf. Proposition 2.4 in \cite{VivekPal}.
\end{proof}

\subsection{Petersson norm} For a positive integer $N$, let $S_2(N)$ be the space of weight $2$ cuspidal holomorphic modular forms for the congruence subgroup $\Gamma_0(N)$ acting on $\hfrak$. Given $f\in S_2(N)$, its Fourier expansion is $f(z)=a_1(f)q+a_2(f)q^2+...$ where $q=\exp(2\pi i z)$ and the numbers $a_n(f)$ are the Fourier coefficients of $f$.  The Petersson norm of $f$ relative to $\Gamma_0(N)$ is defined by
$$
\|f\|_N = \left(\int _{\Gamma_0(N)\backslash \hfrak} |f(z)|^2dx\wedge dy\right)^{1/2}, \quad z=x+iy\in \hfrak.
$$
The norm depends on the choice of $N$ in the following sense:  If $N|M$ and $f\in S_2(N)$, then we certainly have $f\in S_2(M)$, and $\|f\|_M^2=[\Gamma_0(N):\Gamma_0(M)]\cdot \|f\|_N^2$.

We need some additional notation. For an elliptic curve $E$ over $\Q$ of conductor $N$ we denote by $f_E\in S_2(N)$ the Hecke newform attached to $E$ by the modularity theorem, normalized by $a_1(f_E)=1$. The modular form $f_E$ is characterized by the following property: If $p$ is a prime of good reduction for $E$ and we define $a_p(E)=p+1-\# E(\F_p)$, then $a_p(f_E)=a_p(E)$. For a fundamental discriminant $D$, let $\Pcal(D,N)$ be the set of primes $p$ with $p|D$ and $p\nmid 2N$. 
\begin{lemma}[Variation of the Petersson norm under quadratic twist] \label{LemmaNorm} Let $E$ be an elliptic curve over $\Q$ and let $D$ be a fundamental discriminant. Let $N$ and $N^{(D)}$ be the conductors of $E$ and $E^{(D)}$ respectively, and assume that $N|N^{(D)}$. Then $\|f_{E^{(D)}}\|_{N^{(D)}}^2/\|f_E\|_N^2\in \Q^\times$ and we have
$$
v_2(\|f_{E^{(D)}}\|_{N^{(D)}}^2/\|f_E\|_N^2)+1 \ge \sum_{p\in \Pcal(D,N)} v_2\left((p-1)(p+1-a_p(E))(p+1+a_p(E))\right).
$$
\end{lemma}
\begin{proof} The quadratic Dirichlet character attached to $D$ has conductor $|D|$. The result follows from the precise formula given in Theorem 1 of \cite{Delaunay} when one only keeps the contribution of $p=2$ and  the  primes $p\in \Pcal(D,N)$ ---the product of the latter primes is denoted by $D_1$ in \emph{loc. cit.}
\end{proof}
We remark that the terms $(p-1)(p+1-a_p(E))(p+1+a_p(E))$ have a clear conceptual origin; they come from Euler factors of the imprimitive symmetric square $L$-function $L(\mathrm{Sym}^2 f_E,s)$ that are removed by twisting, and $L(\mathrm{Sym}^2 f_E,2)$ is (up to a mild factor) equal to $\|f_E\|^2_N$. See \cite{Zagier, Delaunay, Watkins}.

\subsection{Manin constant}\label{SecManin} Given an elliptic curve $E$ over $\Q$ of conductor $N$, we have that $\phi_E^*\omega_E$ is a regular differential on $X_0(N)=\Gamma_0(N)\backslash \hfrak\cup\{\mbox{cusps}\}$. More precisely
\begin{equation}\label{Eqnpullback}
\phi_E^*\omega_E=2\pi i c_Ef_E(z)dz
\end{equation}
where $c_E$ is a rational number uniquely defined up to sign. We assume that the signs of $\phi_E$ and $\omega_E$ are chosen such that $c_E>0$. It follows from \eqref{EqnFH} and  \eqref{Eqnpullback} that (cf. \cite{Szpiro90, SilvermanE})
\begin{equation}\label{EqnModDeg} m_E=4\pi^2 c_E^2\|f_E\|_N^2\exp(2h(E)).
\end{equation}

 The quantity $c_E$ is called the Manin constant, and a fundamental fact is
\begin{lemma} [cf. \cite{Edixhoven}] \label{LemmaEdix}The Manin constant $c_E$ is an integer.
\end{lemma}
We recall that Manin \cite{Manin} conjectured that if $E$ is a strong Weil curve in the sense that $m_E$ is minimal within the isogeny class of $E$, then $c_E=1$. See \cite{Mazur, AU, ARS, Cesnavicius} and the references therein.

%%
%%

%%%%%%%%%%%%%%%%%%%%%%%%%%%%%%%%%%%%%%
%%%%%%%%%%%%%%%%%%%%%%%%%%%%%%%%%%%%%%
%%%%%%%%%%%%%%%%%%%%%%%%%%%%%%%%%%%%%%
%%%%%%%%%%%%%%%%%%%%%%%%%%%%%%%%%%%%%%
%%%%%%%%%%%%%%%%%%%%%%%%%%%%%%%%%%%%%%
%%%%%%%%%%%%%%%%%%%%%%%%%%%%%%%%%%%%%%

\section{Consequences for Watkins' conjecture}

\begin{lemma}\label{LemmaLowerBd} Let $E$ be an elliptic curve over $\Q$ of conductor $N$ and suppose that $E$ has minimal conductor among its quadratic twists. Let $D$ be a fundamental discriminant.  Then
$$
v_2(m_{E^{(D)}}) \ge v_2(m_E/c_E^2)-4+  \sum_{p\in \Pcal(D,N)} v_2\left((p-1)(p+1-a_p(E))(p+1+a_p(E))\right).
$$
\end{lemma}
\begin{proof} Applying \eqref{EqnModDeg} to $E$ and $E^{(D)}$ we find
$$
\frac{m_{E^{(D)}}}{m_E}   = \frac{c_{E^{(D)}}^2}{c_E^2}\cdot \frac{\|f_{E^{(D)}}\|_{N^{(D)}}^2}{\|f_E\|_N^2}\cdot  \delta(E^{(D)},E).
$$
The result follows from lemmas  \ref{LemmaTwist},  \ref{LemmaNorm}, and \ref{LemmaEdix}.
\end{proof}
\begin{proposition}\label{PropLowerBd} Let $E$ be an elliptic curve over $\Q$ of conductor $N$ with non-trivial rational $2$-torsion and suppose that $E$ has minimal conductor among its quadratic twists. Let $D$ be a fundamental discriminant. We have $v_2(m_{E^{(D)}})\ge 3\omega(D) +v_2(m_E/c_E^2)- (7+3\omega(N))$.
\end{proposition}
\begin{proof}  As $E(\Q)[2]$ is non-trivial and it maps injectively into $E(\F_p)$ for every prime $p\nmid 2N$, we have $p+1\equiv a_p(E)\bmod 2$ for these primes. We get $v_2(m_{E^{(D)}})\ge v_2(m_E/c_E^2)-4+3\cdot \#\Pcal(D,N)$ from  Lemma \ref{LemmaLowerBd}, and the result follows from $\#\Pcal(D,N)\ge \omega(D)-\omega(2N)\ge \omega(D)-\omega(N)-1$.
\end{proof}
The following upper bound for the Mordell-Weil rank is standard and it comes from a bound for a $2$-isogeny Selmer rank (cf. Section X.4 in \cite{SilvermanAEC}; see also \cite{ALP}).
\begin{lemma}\label{LemmaSelmerBd} Let $E$ be an elliptic curve over $\Q$ of conductor $N$  with non-trivial rational $2$-torsion. Then $\rk E(\Q)\le 2\omega(N)-1$.
\end{lemma}
\begin{proof}[Proof of Theorem \ref{ThmMain}]   Since $E^{(D)}[2]\simeq E[2]$ as Galois modules and $E$ has non-trivial rational $2$-torsion, we can use Lemma \ref{LemmaSelmerBd} for $E^{(D)}$, which gives
$$
\rk E^{(D)}(\Q)\le 2\omega(N^{(D)})-1\le 2(\omega(D)+\omega(N))-1.
$$
If Watkins' conjecture fails for $E^{(D)}$, then Proposition \ref{PropLowerBd} would give
$$
2(\omega(D)+\omega(N)) - 1\ge v_2(m_{E^{(D)}})+1\ge 3\omega(D)+v_2(m_E/c_E^2)-6-3\omega(N).
$$
This is not possible when $\omega(D)\ge 6 +5\omega(N)-v_2(m_E/c_E^2)$.
\end{proof}

%%%%%%%%%%%%%%%%%%%%%%%%%%%%%%%%%%%%%%
%%%%%%%%%%%%%%%%%%%%%%%%%%%%%%%%%%%%%%
%%%%%%%%%%%%%%%%%%%%%%%%%%%%%%%%%%%%%%
%%%%%%%%%%%%%%%%%%%%%%%%%%%%%%%%%%%%%%
%%%%%%%%%%%%%%%%%%%%%%%%%%%%%%%%%%%%%%
%%%%%%%%%%%%%%%%%%%%%%%%%%%%%%%%%%%%%%

\section{Acknowledgments}

The first author was supported by a Carroll L. Wilson Award and the MIT International Science and Technology Initiatives (MISTI) during an academic visit to Pontificia Universidad Cat\'olica de Chile. The second author was supported by FONDECYT Regular grant 1190442.

%
%
%

%%%%%%%%%%%%%%%%%%%%%%%%%%%%%%%%%%%%%%
%%%%%%%%%%%%%%%%%%%%%%%%%%%%%%%%%%%%%%
%%%%%%%%%%%%%%%%%%%%%%%%%%%%%%%%%%%%%%

\end{document}